\documentclass[12pt]{amsart}
\usepackage{amscd,amssymb,amsthm,amsmath,amssymb,textcomp,mathrsfs,mathtools,wasysym, longtable}
\usepackage[matrix,arrow,curve]{xy}
\usepackage{supertabular}


\newtheorem{theorem}{Theorem}[section]

\newtheorem{lemma}[theorem]{Lemma}

\newtheorem*{corollary*}{Corollary}
\newtheorem*{proposition*}{Proposition}
\newtheorem*{maincorollary*}{Main Corollary}

\newtheorem*{conjecture*}{Conjecture}

\newtheorem*{problem*}{Problem}

\newtheorem*{theorem*}{Theorem}
\newtheorem*{maintheorem*}{Main Theorem}
\newtheorem*{auxiliarytheorem*}{Auxiliary Theorem}
\theoremstyle{definition}

\newtheorem*{example*}{Example}

\theoremstyle{remark}
\newtheorem{remark}[theorem]{Remark}
\newtheorem*{remark*}{Remark}

\makeatletter\@addtoreset{equation}{section} \makeatother

\usepackage{xcolor}

\author{Ivan Cheltsov and Jihun Park}

\title{Degree of irrationality of Fano threefold hypersurfaces}

\address{ \emph{Ivan Cheltsov}\newline \textnormal{School of Mathematics, University of Edinburgh, Edinburgh, Scotland
\newline
\texttt{I.Cheltsov@ed.ac.uk}}}

\address{ \emph{Jihun Park}\newline \textnormal{Center for Geometry and Physics, Institute for Basic Science, Pohang, Korea \newline
Department of Mathematics, POSTECH, Pohang, Korea \newline
\texttt{wlog@postech.ac.kr}}}

\begin{document}

\begin{abstract}
We study degree of irrationality of quasismooth anticanonically embedded weighted Fano 3-fold hypersurfaces that have terminal
singularities.
\end{abstract}

\maketitle

Throughout this paper, all varieties are assumed to be projective and defined over~$\mathbb{C}$.

\section{Introduction}
\label{section:intro}

Let $a_1, \ldots, a_5$ be positive integers with $a_1\leqslant a_2\leqslant a_3\leqslant a_4\leqslant a_5$.
Let $X_{d}$ be a quasismooth and well-formed (\cite{IF00}) hypersurface of degree $d$ in  the weighted projective space $\mathbb{P}(a_1,a_{2},a_{3},a_{4},a_{5})$.
Suppose that
$$
d<a_1+a_{2}+a_{3}+a_{4}+a_{5}.
$$
and let $I=a_1+a_{2}+a_{3}+a_{4}+a_{5}-d$.
Then $X_d$ is a Fano threefold with quotient singularities.
Note that there exists infinitely many possibilities for the septuple  $(a_1,a_{2},a_{3},a_{4},a_{5},d,I)$.
If~$I=1$, all of them have been found in  \cite{JoKo01}.

Suppose,~in~addition, that the singularities of the 3-fold hypersurface $X_d$ are~terminal.
Then it follows from the Birkar's boundedness results (\cite{Birkar19,Birkar21}) that
there are finitely many possibilities for  $(a_1,a_{2},a_{3},a_{4},a_{5},d, I)$.
In fact, there are $130$ such possibilities, excluding weighted projective spaces,
that define $130$ families of weighted hypersurfaces (\cite{BrownSuzuki1,BrownSuzuki2, IF00}).
Fano 3-folds in these families have been studied in \cite{AbbanCheltsovPark,Cheltsov2007,Cheltsov2008,Cheltsov2009,CheltsovPark2006,CheltsovPark2009,CheltsovPark2017,CPR,KimOkadaWon,Okada2019,Ryder2006,Ryder2009}.
In particular, we know that $X_d$ is irrational if $I=1$, and we know all deformation families with $I\geqslant 2$
such that their general members are irrational.
In the following, we will label these $130$ deformation families by \textnumero 1, \textnumero 2, \textnumero 3, $\ldots$, \textnumero 130 as in \cite{AbbanCheltsovPark,CheltsovPark2017,CPR}.

\begin{remark*}
If $X_d$ is smooth, then it belongs to one of the following deformation families:
\begin{itemize}
\item[\textnumero 1:] $(a_1,a_{2},a_{3},a_{4},a_{5},d,I)=(1,1,1,1,1,4,1)$, i.e. $X_d$ is a quartic 3-fold,
\item[\textnumero 3:] $(a_1,a_{2},a_{3},a_{4},a_{5},d,I)=(1,1,1,1,3,6,1)$, i.e. $X_d$ is a sextic double solid.
\item[\textnumero 96:] $(a_1,a_{2},a_{3},a_{4},a_{5},d,I)=(1,1,1,1,1,3,2)$, i.e. $X_d$ is a cubic 3-fold,
\item[\textnumero 97:] $(a_1,a_{2},a_{3},a_{4},a_{5},d,I)=(1,1,1,1,1,3,2)$, i.e. $X_d$ is a quartic double solid,
\item[\textnumero 98:] $(a_1,a_{2},a_{3},a_{4},a_{5},d,I)=(1,1,1,2,3,6,2)$, i.e. $X_d$ is a double Veronese cone,
\item[\textnumero 104:] $(a_1,a_{2},a_{3},a_{4},a_{5},d,I)=(1,1,1,1,1,2,3)$, i.e. $X_d$ is a quadric 3-fold.
\end{itemize}
All members of the remaining $124$ deformation families are singular,
and their singularities are described~in~\cite{BrownSuzuki1,BrownSuzuki2, CPR}.
Note that $I=1$ for the families  \textnumero 1, \textnumero 2, \textnumero 3, $\ldots$, \textnumero 94, \textnumero 95.
\end{remark*}

In this paper, we study the degree of irrationality of Fano 3-folds in these $130$ families.
Recall from \cite{MoHe82} that the degree of irrationality $\mathrm{d}(V)$ of a variety $V$ is
the smallest possible degree of the generically finite dominant rational map $V\dasharrow\mathbb{P}^{n}$, where $n=\mathrm{dim}(V)$. Then
\begin{center}
the variety $V$ is rational $\iff$ $\mathrm{d}(V)=1$.
\end{center}
If $I\geqslant 2$, then $d<3a_{5}$, so the projection
$\mathbb{P}(1,a_{2},a_{3},a_{4},a_{5})\dasharrow\mathbb{P}(1,a_{2},a_{3},a_{4})$
induces a dominant rational map
$X_{d}\dasharrow\mathbb{P}(1,a_{2},a_{3},a_{4})$ such that it is either birational or generically two-to-one,
so that $\mathrm{d}(X_d)\in\{1,2\}$ in this case.
For $I=1$, we prove

\begin{maintheorem*}
Suppose that $I=1$. Then the following assertions hold:
\begin{enumerate}
\item $\mathrm{d}(X_{d})\in\{2,3\}$;
\item $\mathrm{d}(X_{d})=2$ if the 3-fold $X_d$ is not contained in the deformation families   \textnumero 1, \textnumero 19, \textnumero 28, \textnumero 39, \textnumero 49, \textnumero 59, \textnumero 66, \textnumero 84;
\item $\mathrm{d}(X_{d})=3$, if the 3-fold $X_d$ is a general member of one of the eight deformation families
\textnumero 1, \textnumero 19, \textnumero 28, \textnumero 39, \textnumero 49,
\textnumero 59, \textnumero 66, \textnumero 84.
\end{enumerate}
\end{maintheorem*}

\begin{conjecture*}
Suppose that $X_d$ is any quasismooth member in one of the eight deformation families
\textnumero 1, \textnumero 19, \textnumero 28, \textnumero 39, \textnumero 49, \textnumero 59, \textnumero 66, \textnumero 84.
Then $\mathrm{d}(X_{d})=3$.
\end{conjecture*}

If $X_d$ is contained in one of the families
\textnumero 1, \textnumero 19, \textnumero 28, \textnumero 39, \textnumero 49, \textnumero 59, \textnumero 66, \textnumero 84,
then the septuple  $(a_1,a_{2},a_{3},a_{4},a_{5},d,I)$ is described in the following table below.

\begin{center}
\renewcommand\arraystretch{1.1}
\begin{tabular}{|c||c|c|c|c|c|c|c|}
\hline \textnumero & $a_1$ & $a_2$ & $a_3$ & $a_4$ & $a_5$ & $d$ & $I$\\
\hline\hline
1 & $1$ & $1$ & $1$ & $1$ & $1$ & $4$ & $1$\\
\hline
19 & $1$ & $2$ & $3$ & $3$ & $4$ & $12$ & $1$\\
\hline
28 & $1$ & $3$ & $3$ & $4$ & $5$ & $15$ & $1$ \\
\hline
39 & $1$ & $3$ & $4$ & $5$ & $6$ & $18$ & $1$\\
\hline
49 & $1$ & $3$ & $5$ & $6$ & $7$ & $21$ & $1$\\
\hline
59 & $1$ & $3$ & $6$ & $7$ & $8$ & $24$ & $1$\\
\hline
66 &  $1$ & $5$ & $6$ & $7$ & $9$ & $27$ & $1$\\
\hline
84 & $1$ & $7$ & $8$ & $9$ & $12$ & $36$ & $1$\\
\hline
\end{tabular}
\end{center}

\medskip

It would be interesting to study degree of irrationality of rationally connected threefolds in a more general setting.
If a rationally connected 3-fold is birational to a conic bundle,
its degree of irrationality is $1$ or $2$.
Similarly, if a rationally connected 3-fold is birational to a fibration into del Pezzo surfaces,
it is birational to an elliptic fibration with a section,
which implies that the degree of irrationality of the threefold is either $1$ or $2$.

\begin{proposition*}
Let $V$ be a smooth Fano 3-fold that is not a quartic 3-fold. Then $\mathrm{d}(V)$  is either $1$ or $2$.
\end{proposition*}

\begin{proof}
We may assume that $V$ is not birational neither to a conic bundle
nor to an elliptic fibration with a section.
Then $V$ is sextic double solid. This gives $\mathrm{d}(V)\leqslant 2$.
\end{proof}

Applying Minimal Model Program, we see that every rationally connected threefold whose degree of irrationality is larger than $2$
is birational to a Fano 3-fold with terminal $\mathbb{Q}$-factorial singularities that has Picard rank $1$,
which is birational neither to a conic bundle nor to a fibration into del Pezzo surfaces.
Such a Fano 3-fold is  said to be  \emph{solid}~(\cite{AO2018}).
It would be interesting to study the degree of irrationality of solid Fano 3-folds.

Let us describe the structure of this paper.
In Sections~\ref{section:1}, \ref{section:19-28}, and \ref{section:the other}, we study weighted Fano 3-fold hypersurfaces in the deformation families
\textnumero 1, \textnumero 19, \textnumero 28, \textnumero 39, \textnumero 49, \textnumero 59, \textnumero 66, \textnumero 84.
Then, in Section~\ref{section:main}, we prove out Main Theorem.

From now on, we denote by $\mathrm{Aut}(X)$  the group of automorphisms of a threefold $X$ and by $\mathrm{Bir}(X)$ the group of  birational automorphisms of $X$.
We use coordinates $x$, $y$, $z$, $t$, $w$ for $\mathbb{P}(a_1,a_{2},a_{3},a_{4},a_{5})$ with $\mathrm{wt}(x)=a_1$, $\mathrm{wt}(y)=a_2$, $\mathrm{wt}(z)=a_3$, $\mathrm{wt}(t)=a_4$, $\mathrm{wt}(w)=a_5$.

\medskip

\textbf{Acknowledgements.}
Cheltsov has been supported by EPSRC Grant EP/V054597/1. Park has been supported by IBS-R003-D1, Institute for Basic Science in Korea.

\section{Family \textnumero 1}
\label{section:1}

Let $X_4$ be a smooth quartic 3-fold in $\mathbb{P}^4$.
Then $X_4$ is irrational (\cite{IsMa71}).
Thus, considering a projection  $X_4\dasharrow\mathbb{P}^{3}$ from a point in $X_4$, we get  $\mathrm{d}(V)\in\{2,3\}$.

\begin{lemma}
\label{lemma:1} If $X_4$ is general, then $\mathrm{d}(X_4)=3$.
\end{lemma}

\begin{proof}
Suppose that $\mathrm{d}(X_4)=2$.
Then there exists a rational map $X_4\dasharrow\mathbb{P}^3$ that is generically two-to-one.
This map induces a birational involution $\tau$ of $X_4$. The involution~$\tau$ is biregular since
$$
\mathrm{Bir}(X_4)=\mathrm{Aut}(X_4).
$$ by \cite{IsMa71}. However, if $X_4$ is general, the group $\mathrm{Aut}(X_4)$ is trivial (\cite{MaMon64}).
\end{proof}

We expect that $\mathrm{d}(X_4)=3$.
Indeed, suppose $\mathrm{d}(X_4)=2$.
As in the proof of Lemma~\ref{lemma:1}, we obtain a biregular involution $\tau$ of $X_4$
such that the quotient $Y=X_4\slash\tau$ is rational.
Then the fixed locus of the involution $\tau$ is either
\begin{itemize}
\item  a smooth quartic surface,
\item  or a disjoint union of a smooth plane quartic curve and four different points.
\end{itemize}
In the former case, the threefold $Y$ is a double cover of $\mathbb{P}^3$ ramified along a smooth quartic,
which is irrational (\cite{Vo88}).
Therefore, we conclude that $\tau$ fixes disjoint union of smooth plane quartic curve and four points.
Now, changing coordinates on $\mathbb{P}^4$, we may assume that
$$
X_4=\big\{h_{4}(x,y,z)+t^{2}a_{2}(x,y,z)+twb_{2}(x,y,z)+w^{2}c_{2}(x,y,z)+g_{4}(t,w)=0\big\}\subset\mathbb{P}^4
$$
while $\tau$ acts as
$$
\tau(x:y:z:t:w)=(x:y:z:-t:-w),
$$
where
$h_{i}$, $a_{i}$, $b_{i}$, $c_{i}$, $g_{i}$ are homogeneous polynomials of degree $i$.
Then $\tau$ fixes the smooth
irreducible curve
$$
C=\big\{h_{4}(x,y,z)=0,t=0,w=0\big\}\subset\mathbb{P}^4
$$
and four points $O_{1}, O_{2}, O_{3}$, $O_{4}$ that are cut out on $X_4$ by $x=y=z=0$.

Let $\eta\colon X_4\to Y$ be the quotient map.
Then $Y$ is singular along the smooth curve $\eta(C)$,
and $Y$ has $4$ singular points
$\eta(O_{1})$, $\eta(O_{2})$, $\eta(O_{3})$, and $\eta(O_{4})$,
which are cyclic quotient singular points of type ${\frac{1}{2}}(1,1,1)$.
Let $\mathcal{H}$ be the pencil on $Y$ generated by the images of the surfaces cut out on $X_4$ by $t=0$ and $w=0$, respectively.
Then we have the following commutative diagram:
$$
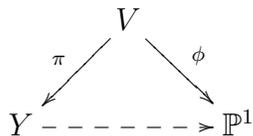
\xymatrix{
&V\ar@{->}[ld]_{\pi}\ar@{->}[rd]^{\phi}&\\%
Y\ar@{-->}[rr]&&\mathbb{P}^1}
$$
where $\pi$ is the blow up of the curve $\eta(C)$,
the map $Y\dasharrow\mathbb{P}^1$ is given by the~pencil $\mathcal{H}$,
and~$\phi$ is a Mori fibre space whose general fiber is a smooth del Pezzo surfaces of degree~$2$.
Now, one can try to apply the technique from \cite{AbbanKrylov,Krylov,Okada} to show that $Y$ is irrational.

\section{Families \textnumero 19 and \textnumero 28}
\label{section:19-28}

Before we proceed, let us introduce a lemma which is useful in handling automorphisms for Families \textnumero 19 and \textnumero 28.
\begin{lemma}\label{lemma:5-point}
Let $\Sigma$ be a subset in $\mathbb{P}^1$ that consists of $5$ distinct points in general position.
Then the group of automorphisms of the pair $(\mathbb{P}^1, \Sigma)$ is trivial.
\end{lemma}

\begin{proof}
This is a well-known fact. See, for example, \cite[\S~6.4]{DolgachevIskovskikh} or \cite{duPlessisWall}.
\end{proof}

Let $Y$ be a general quasismooth hypersurface in the family  \textnumero 19.
Then $Y$  is given  by quasihomogeneous equation of degree $12$
$$
f:=w^3+w^2f_{4}(x,y,z,t)+wf_{8}(x,y,z,t)+f_{12}(x,y,z,t)=0
$$
in  $\mathbb{P}(1,2,3,3,4)$,
where each $f_k$ is a quasihomogeneous polynomial of degree $k$ in $x, y, z, t$.
We now change the coordinates in the following way:
\begin{enumerate}
\item by using the following type of  coordinate change   $$w\mapsto w+*y^2+*yx^2+*xz+*xt,$$
we may assume that $f$ has $wy^4$ but none of $y^6$,  $y^5x^2$, $y^4xz$, $y^4xt$;
\item   by using the following type of  coordinate change   $$y\mapsto y+*x^2,$$
we may assume that  $f$ does not have  $wy^3x^2$;
\item  by using the following type of  coordinate change   $$z\mapsto z+*xy+*x^3,$$
we may assume  that $f$ has  neither $t^3xy$ nor $t^3x^3$;
\item  by using the following type of  coordinate change   $$t\mapsto t+*xy+*x^3,$$
we may assume  that $f$ has  neither $z^3xy$ nor $z^3x^3$,
\end{enumerate}
where $*$ means an appropriate constant.
Note that our choice of coordinates implies that the sections by $w=0$ and $y=0$ are preserved by $\mathrm{Aut}(Y)$.
Furthermore, the section by $x=0$  is also preserved by $\mathrm{Aut}(Y)$, since the $0$-dimensional space $|\mathcal{O}_{Y}(1)|$ is invariant.
These imply that all members of $\mathrm{Aut}(Y)$ are given by
$$
[x:y:z:t;w]\mapsto [x: A_1y:   \alpha_1z+ \alpha_2t: \alpha_3z+ \alpha_4t:A_2w],
$$
where $A_1$ and $A_2$ are non-zero constants,  and $\alpha_i$'s  are constants such that the matrix
\[
\begin{pmatrix}
\alpha_1 & \alpha_2  \\
\alpha_3 & \alpha_4
\end{pmatrix}\]
is invertible. Observe also that
$$
f_{12}(x,y,z,t)=y^3g_{6}(z,t)+y^2h_8(x,z,t)+yh_{10}(x,z,t)+h_{12}(x,z,t),
$$
where $g_6$ is a quasihomogeneous polynomial of degree $6$ in $z,t$,
and $h_k$'s are quasihomogeneous polynomials of degree $k$ in $x,z,t$.
The hypersurface $Y$ has four distinct points
$$
g_{12}(z,t)=0
$$
on the curve $C$ defined by $x=y=w=0$.
Moreover, since $Y$ is general, we may assume that $g_{6}(z,t)=0$ gives two distinct points on the curve $C$,
and hence
$$
g_{12}(z,t)g_6(z,t)=0
$$
defines $6$ distinct points $\Sigma$ in $C\cong \mathbb{P}^1$ which are in general position.
Every automorphism of $Y$ induces an automorphism of $C$ that leaves $\Sigma$ invariant.
Then Lemma~\ref{lemma:5-point} implies that $\alpha_2=\alpha_3=0$ and $\alpha_1=\alpha_4$.
Consequently,  every automorphism of $Y$ must be induced by
$$
[x:y:z:t:w]\mapsto [x: a_1y: a_2z: a_2t:a_3w]
$$
where $a_i$'s are non-zero constants.

Let $W$ be a general quasismooth hypersurface in the family  \textnumero 28.
The hypersurface $W$  is given  by quasihomogeneous equation of degree $15$
$$
w^3+w^2f_{5}(x,y,z,t)+wf_{10}(x,y,z,t)+f_{15}(x,y,z,t)=0
$$
in  $\mathbb{P}(1,3,3,4,5)$,
where $f_k(x,y,z,t)$ is a quasihomogeneous polynomial of degree $k$ in $x, y, z, t$.
By using the coordinate change
$$w\mapsto w+\frac{1}{3}f_{5}(x,y,z,t)$$
we may assume that $W$ is defined by quasihomogeneous equation of degree $15$
$$
w^3+wf_{10}(x,y,z,t)+f_{15}(x,y,z,t)=0
$$
in  $\mathbb{P}(1,3,3,4,5)$.
Since $W$ is quasismooth, we may further change coordinates $z$ and $y$ such that $f_{15}$ has $t^3z$ but not~$t^3y$.
Then, we change the coordinates in the following way:
\begin{enumerate}
\item use the following type of  coordinate change   $$z\mapsto z+*x^3$$
 in such a way that  $f_{15}$ does not have $t^3x^3$;
\item use the following type of  coordinate change   $$y\mapsto y+*x^3$$
 in such a way that  $f_{15}$ has $y^5$ but not $y^4x^3$;
\item  use the following type of  coordinate change   $$t\mapsto t+*zx+*yx+*x^4$$
 in such a way that none of $t^2z^2x$, $t^2zyx$, $t^2zx^4$ appear in $f_{15}$,
\end{enumerate}
where $*$ means a suitable constant.
As above, we see that each member in $\mathrm{Aut}(W)$ is given by
$$
[x:y:z:t;w]\mapsto [x: \alpha_1y+ \alpha_2z:  \alpha_3y+ \alpha_4z:A_1t:A_2w]
$$
where $A_i$'s are non-zero constants, and $\alpha_i$'s  are constants such that the matrix
\[
\begin{pmatrix}
\alpha_1 & \alpha_2  \\
\alpha_3 & \alpha_4
\end{pmatrix}\]
is invertible. Observe that the hypersurface $W$ has five distinct points on the curve $C$ defined by $x=t=w=0$,
and every automorphism of $W$ induces an automorphism of $C$ that permutes these five points.
Since $W$ is general, we may assume that these points are in general position.
Thus, applying Lemma~\ref{lemma:5-point}, we see that $\alpha_2=\alpha_3=0$ and $\alpha_1=\alpha_4$.
Consequently,  every automorphism of $W$ must be induced by
$$
[x:y:z:t:w]\mapsto [x:a_1y:a_1z:a_3t:a_4w]
$$
where $a_i$'s are non-zero constants.

The following table shows all monomials that appear in the quasihomogenous polynomial for each of the families \textnumero 19 and \textnumero 28 after required coordinate changes.

\begin{center}
\begin{longtable}{|c||c|c|}
\hline \textnumero & $d$ & monomials of degree $d$ \\
\hline\hline
19 &
  \begin{minipage}[m]{.03\linewidth}\setlength{\unitlength}{.20mm}
  \smallskip
  \begin{center}
$12$
\end{center}
\smallskip
\end{minipage}
&
  \begin{minipage}[m]{.87\linewidth}\setlength{\unitlength}{.20mm}
  \smallskip
  \begin{center}
$w^{3}$, $zt^{3}$, $z^{2}t^{2}$, $z^{3}t$, $yt^{2}w$, $yztw$, $yz^{2}w$,  $y^2w^2$, $y^{3}t^{2}$, $y^{3}zt$, $y^{3}z^{2}$, $y^{4}w$,
 $xyzt^{2}$, $xyz^{2}t$,  $xy^{2}tw$, $xy^{2}zw$, $xzw^2$, $xtw^2$,
 $x^{2}t^{2}w$, $x^{2}ztw$, $x^{2}z^{2}w$,  $x^{2}y^{2}t^{2}$, $x^{2}y^{2}zt$, $x^{2}y^{2}z^{2}$,
$x^{3}zt^{2}$, $x^{3}z^{2}t$, $x^{3}ytw$, $x^{3}yzw$, $x^{3}y^{3}t$, $x^{3}y^{3}z$,
 $x^{4}yt^{2}$, $x^{4}yzt$, $x^{4}yz^{2}$, $x^{4}y^{2}w$, $x^{4}y^{4}$, $x^4w^2$, $x^{5}tw$, $x^{5}zw$, $x^{5}y^{2}t$, $x^{5}y^{2}z$,
 $x^{6}t^{2}$, $x^{6}zt$, $x^{6}z^{2}$, $x^{6}yw$, $x^{6}y^{3}$, $x^{7}yt$, $x^{7}yz$, $x^{8}w$,
$x^{8}y^{2}$, $x^{9}t$, $x^{9}z$, $x^{10}y$, $x^{12}$
\end{center}
\smallskip
\end{minipage}\\
\hline
28 &
 \begin{minipage}[m]{.03\linewidth}\setlength{\unitlength}{.20mm}
  \smallskip
  \begin{center}
$15$
\end{center}
\smallskip
\end{minipage}
&
  \begin{minipage}[m]{.87\linewidth}\setlength{\unitlength}{.20mm}
  \smallskip
  \begin{center}
 $w^{3}$, $zt^{3}$, $z^{2}tw$,  $yt^{3}$, $yztw$, $yz^{4}$, $y^{2}tw$, $y^{2}z^{3}$, $y^{3}z^{2}$, $y^{4}z$,
 $xz^{3}w$, $xyz^{2}w$, $xy^{2}t^{2}$, $xy^{2}zw$, $xy^{3}w$,
$x^{2}t^{2}w$, $x^{2}z^{3}t$,  $x^{2}yz^{2}t$, $x^{2}y^{2}zt$, $x^{2}y^{3}t$,
 $x^{3}ztw$,  $x^{3}ytw$, $x^{3}yz^{3}$, $x^{3}y^{2}z^{2}$, $x^{3}y^{3}z$,
 $x^{4}z^{2}w$, $x^{4}yt^{2}$, $x^{4}yzw$, $x^{4}y^{2}w$,  $x^{5}z^{2}t$, $x^{5}yzt$, $x^{5}y^{2}t$,
 $x^{6}tw$, $x^{6}z^{3}$, $x^{6}yz^{2}$, $x^{6}y^{2}z$, $x^{6}y^{3}$, $x^{7}t^{2}$, $x^{7}zw$, $x^{7}yw$, $x^{8}zt$, $x^{8}yt$,
 $x^{9}z^{2}$, $x^{9}yz$, $x^{9}y^{2}$, $x^{10}w$, $x^{11}t$, $x^{12}z$, $x^{12}y$, $x^{15}$
   \end{center}
\smallskip
\end{minipage}\\
\hline
\end{longtable}
\end{center}
Using this, it is not difficult to see that the automorphism group of a general hypersurface in each of the families \textnumero 19 and \textnumero 28 is trivial.
In particular, it contains no involutions.

Suppose that for a general hypersurface $X$ in the families \textnumero 19 and \textnumero 28, $\mathrm{d}(X)=2$. Then $\mathrm{Bir}(X)$ has an involution.
However, since  $X$ is birationally super-rigid (\cite{CheltsovPark2017}), we have
$$
\mathrm{Bir}(X)=\mathrm{Aut}(X),
$$
and hence  $\mathrm{Aut}(X)$ has an involution. This is a contradiction. Therefore, $\mathrm{d}(X)=3$.

\section{Families \textnumero 39, \textnumero 49, \textnumero 59, \textnumero 66, \textnumero 84}
\label{section:the other}

Let $X_{d}$ be a quasismooth and well-formed hypersurface of degree $d$ in  the weighted projective space $\mathbb{P}(a_1,a_{2},a_{3},a_{4},a_{5})$.
Suppose that $X_d$ belongs to one of the families \textnumero 39, \textnumero 49, \textnumero 59, \textnumero 66, \textnumero 84. Since the hypersurface $X_d$ is birationally super-rigid (\cite{CheltsovPark2017}), we have
$$
\mathrm{Bir}(X_d)=\mathrm{Aut}(X_d).
$$
It follows from the same argument in Section~\ref{section:19-28} that in order to have $\mathrm{d}(X_d)=3$, it is enough to show that
$\mathrm{Aut}(X_d)$ is trivial.

We have $d=3a_5$ and the hypersurface is defined by quasihomogeneous equation
\[f(x,y,z,t, w):=w^3+w^2f_{a_5}(x,y,z,t)+wf_{2a_5}(x,y,z,t)+f_{d}(x,y,z,t)=0,\]
where $f_k(x,y,z,t)$ is a quasihomogeneous polynomial of degree $k$ in $x, y, z, t$.

By using the coordinate change $$w\mapsto w+\frac{1}{3}f_{a_5}(x,y,z,t),$$ we may assume that $f_{a_5}(x,y,z,t)=0$, that is, the hypersurface is
 is defined by quasihomogeneous equation
 \begin{equation}
 \label{eq:w}
w^3+wf_{2a_5}(x,y,z,t)+f_{d}(x,y,z,t)=0.
\end{equation}

The section by $x=0$  is preserved by $\mathrm{Aut}(X_d)$ since the $0$-dimensional space $|\mathcal{O}_{X_d}(1)|$ is invariant.
Due to the assumption \eqref{eq:w}, the section by $w=0$ is also invariant under the action of  $\mathrm{Aut}(X_d)$.

We now suppose that $X_d$ is general.
Using the monomials in the first column, we perform appropriate coordinate changes described in the third column of Table~\ref{eq:yzt} so that
we might assume that  the monomials in the second column do not appear in the defining equation \eqref{eq:w}.
For instance, if $f= \cdots+at^3y+bt^3x^3+\cdots$ in  the first line of~\textnumero 39, where $a\ne 0$, $b$ are constants, then we perform the coordinate change  $$y\mapsto y+\frac{b}{a}x^3$$ so that we could remove the monomial $t^3x^3$ from $f$.  If $f= \cdots+at^3y+bt^2yzx+ct^2y^2x^2+dt^2yx^5+\cdots$ in  the second line of \textnumero 39, where $a\ne 0$, $b$, $c$, $d$ are constants, then we use the coordinate change
$$t\mapsto t+\frac{1}{3a}\left(bxz+cx^2y+dx^5\right)$$ to get rid of the monomials $t^2yzx$, $t^2y^2x^2$,  $t^2yx^5$ from $f$.

\begin{equation}\label{eq:yzt}
\renewcommand\arraystretch{1.4}
\begin{tabular}{|c||c|c|c|}
\hline \textnumero & $f$ contains & $f$ does not contain  & coordinate changes\\
\hline\hline
39 &  \begin{minipage}[m]{.10\linewidth}\setlength{\unitlength}{.20mm}
  \smallskip
  \begin{center}
$t^3y$

$t^3y$

$z^3w$
\end{center}
\smallskip
\end{minipage}
&
\ \
  \begin{minipage}[l]{.22\linewidth}\setlength{\unitlength}{.20mm}
  \smallskip
  $t^3x^3$

  $t^2yzx$, $t^2y^2x^2$,  $t^2yx^5$

  $z^2wyx$, $z^2wx^4$
\smallskip
\end{minipage}
&
\ \
  \begin{minipage}[l]{.35\linewidth}\setlength{\unitlength}{.20mm}
  \smallskip
  $y\mapsto y+*x^3$

  $t\mapsto t+*xz+*x^2y+*x^5$

  $z\mapsto z+*xy+*x^4$

\smallskip
\end{minipage}\\
\hline
49 &  \begin{minipage}[m]{.10\linewidth}\setlength{\unitlength}{.20mm}
  \smallskip
  \begin{center}
$t^3y$

$y^5t$

$z^4x$
\end{center}
\smallskip
\end{minipage}
&
\ \
  \begin{minipage}[m]{.22\linewidth}\setlength{\unitlength}{.20mm}
  \smallskip
  $t^3x^3$

  $y^5xz$, $y^5x^6$, $y^6x^3$, $y^7$

  $z^3x^3y$,  $z^3x^6$,
\smallskip
\end{minipage}
&
\ \
  \begin{minipage}[l]{.35\linewidth}\setlength{\unitlength}{.20mm}
  \smallskip
   $y\mapsto y+*x^3$

  $t\mapsto t+*xz+*x^3y+*x^6+*y^2$

  $z\mapsto z+*x^2y+*x^5$
\smallskip
\end{minipage}\\
\hline
59 &  \begin{minipage}[m]{.10\linewidth}\setlength{\unitlength}{.20mm}
  \smallskip
  \begin{center}
$t^3y$

$t^3y$

$y^6z$
\end{center}
\smallskip
\end{minipage}
&
\ \
  \begin{minipage}[m]{.22\linewidth}\setlength{\unitlength}{.20mm}
  \smallskip
  $t^3x^3$

  $t^2y^3x$, $t^2yxz$, $t^2yx^7$

  $y^6x^6$, $y^7x^3$, $y^8$
\smallskip
\end{minipage}
&
\ \
  \begin{minipage}[l]{.35\linewidth}\setlength{\unitlength}{.20mm}
  \smallskip
   $y\mapsto y+*x^3$

  $t\mapsto t+*xy^2+*xz+*x^7$

  $z\mapsto z+*y^2+*yx^3+x^6$

\smallskip
\end{minipage}\\
\hline
66 &   \begin{minipage}[m]{.10\linewidth}\setlength{\unitlength}{.20mm}
  \smallskip
  \begin{center}
$t^3z$

$y^4t$

$y^4t$
\end{center}
\smallskip
\end{minipage}
&
\ \
  \begin{minipage}[l]{.22\linewidth}\setlength{\unitlength}{.20mm}
  \smallskip
  $t^3xy$, $t^3x^6$

  $y^4xz$, $y^4x^7$,  $y^5x^2$

  $y^3tx^5$
\smallskip
\end{minipage}
&
\ \
  \begin{minipage}[l]{.35\linewidth}\setlength{\unitlength}{.20mm}
  \smallskip
   $z\mapsto z+*xy+*x^6$

  $t\mapsto t+*xz+*x^2y+*x^7$

  $y\mapsto y+*x^5$

\smallskip
\end{minipage}\\
\hline
84 &  \begin{minipage}[m]{.10\linewidth}\setlength{\unitlength}{.20mm}
  \smallskip
  \begin{center}
$y^4z$

$y^4z$

$t^4$
\end{center}
\smallskip
\end{minipage}
&
\ \
  \begin{minipage}[l]{.22\linewidth}\setlength{\unitlength}{.20mm}
  \smallskip
   $y^4x^8$, $y^5x$

   $y^3zx^7$

   $t^3xz$, $t^3x^2y$, $t^3x^9$
\smallskip
\end{minipage}
&
\ \
  \begin{minipage}[l]{.35\linewidth}\setlength{\unitlength}{.20mm}
  \smallskip
  $z\mapsto z+*xy+*x^8$

  $y\mapsto y+*x^7$

  $t\mapsto t+*xz+*x^2y+*x^9$

\smallskip
\end{minipage}\\
\hline
\end{tabular}
\end{equation}
where $*$ means an appropriate constant.
It then follows  that $\mathrm{Aut}(X_d)$  fixes the sections defined by
$y=0$, by $z=0$, and  by $t=0$.
Consequently,
all members of $\mathrm{Aut}(X_d)$ are given by
$$
[x:y:z:t:w]\mapsto [x:a_1y:a_2z:a_3t:a_4w]
$$
where $a_i$'s are non-zero constants.

The following table shows all the monomials that appear in the quasihomogenous polynomial  \eqref{eq:w}
for each of the families \textnumero 39, \textnumero 49, \textnumero 59, \textnumero  66, \textnumero  84.
Note that all the assumptions of Table~\ref{eq:yzt} for  \eqref{eq:w} are applied to the table below.

\begin{longtable}{|c||c|c|}
\hline \textnumero & $d$ & monomials of degree $d$ \\
\hline\hline
39 &  \begin{minipage}[m]{.03\linewidth}\setlength{\unitlength}{.20mm}
  \smallskip
  \begin{center}
$18$
\end{center}
\smallskip
\end{minipage}
&
  \begin{minipage}[m]{.87\linewidth}\setlength{\unitlength}{.20mm}
  \smallskip
  \begin{center}
$w^{3}$, $z^{2}t^{2}$, $z^{3}w$, $yt^{3}$, $yztw$,  $y^{2}z^{3}$, $y^{3}zt$, $y^{4}w$, $y^{6}$,
 $xz^{3}t$, $xy^{2}tw$, $xy^{3}z^{2}$, $xy^{4}t$,
$x^{2}t^{2}w$,  $x^{2}z^{4}$, $x^{2}yz^{2}t$,  $x^{2}y^{2}zw$, $x^{2}y^{4}z$,
$x^{3}ztw$,  $x^{3}yz^{3}$, $x^{3}y^{2}zt$, $x^{3}y^{3}w$, $x^{3}y^{5}$,
$x^{4}zt^{2}$,
$x^{4}ytw$, $x^{4}y^{2}z^{2}$, $x^{4}y^{3}t$, $x^{5}z^{2}t$,
$x^{5}yzw$, $x^{5}y^{3}z$,
$x^{6}z^{3}$, $x^{6}yzt$, $x^{6}y^{2}w$, $x^{6}y^{4}$, $x^{7}tw$, $x^{7}yz^{2}$, $x^{7}y^{2}t$, $x^{8}t^{2}$, $x^{8}zw$, $x^{8}y^{2}z$,
$x^{9}zt$, $x^{9}yw$, $x^{9}y^{3}$, $x^{10}z^{2}$, $x^{10}yt$, $x^{11}yz$, $x^{12}w$,  $x^{12}y^{2}$, $x^{13}t$, $x^{14}z$, $x^{15}y$, $x^{18}$
  \end{center}
\smallskip
\end{minipage}\\
\hline
49 &  \begin{minipage}[m]{.03\linewidth}\setlength{\unitlength}{.20mm}
  \smallskip
  \begin{center}
$21$
\end{center}
\smallskip
\end{minipage}
&
  \begin{minipage}[m]{.87\linewidth}\setlength{\unitlength}{.20mm}
  \smallskip
  \begin{center}
$w^{3}$, $z^{3}t$, $yt^{3}$, $yztw$, $y^{2}z^{3}$, $y^{3}t^{2}$, $y^{3}zw$, $y^{5}t$,
$xz^{4}$, $xyzt^{2}$, $xyz^{2}w$, $xy^{3}zt$,
$x^{2}t^{2}w$,  $x^{2}yz^{2}t$, $x^{2}y^{2}tw$, $x^{2}y^{3}z^{2}$, $x^{2}y^{4}w$,
  $x^{3}ztw$, $x^{3}y^{2}t^{2}$, $x^{3}y^{2}zw$, $x^{3}y^{4}t$,
 $x^{4}zt^{2}$, $x^{4}z^{2}w$,  $x^{4}y^{2}zt$, $x^{4}y^{4}z$,
$x^{5}z^{2}t$, $x^{5}ytw$, $x^{5}y^{2}z^{2}$, $x^{5}y^{3}w$,
$x^{6}yt^{2}$, $x^{6}yzw$, $x^{6}y^{3}t$,
 $x^{7}yzt$,
$x^{7}y^{3}z$, $x^{8}tw$, $x^{8}yz^{2}$, $x^{8}y^{2}w$, $x^{9}t^{2}$, $x^{9}zw$,
$x^{9}y^{2}t$, $x^{9}y^{4}$,
$x^{10}zt$, $x^{10}y^{2}z$, $x^{11}z^{2}$, $x^{11}yw$,
$x^{12}yt$, $x^{12}y^{3}$, $x^{13}yz$, $x^{14}w$,
$x^{15}t$, $x^{15}y^{2}$, $x^{16}z$, $x^{18}y$, $x^{21}$
\  \end{center}
\smallskip
\end{minipage}\\
\hline
59 &  \begin{minipage}[m]{.03\linewidth}\setlength{\unitlength}{.20mm}
  \smallskip
  \begin{center}
$24$
\end{center}
\smallskip
\end{minipage}
&
  \begin{minipage}[m]{.87\linewidth}\setlength{\unitlength}{.20mm}
  \smallskip
  \begin{center}
$w^{3}$, $z^{4}$, $yt^{3}$, $yztw$, $y^{2}z^{3}$, $y^{3}tw$, $y^{4}z^{2}$,
$y^{6}z$,   $xyz^{2}w$, $xy^{3}zw$,
$xy^{5}w$, $x^{2}t^{2}w$,  $x^{2}yz^{2}t$,  $x^{2}y^{3}zt$,
$x^{2}y^{5}t$,  $x^{3}ztw$, $x^{3}yz^{3}$, $x^{3}y^{2}tw$, $x^{3}y^{3}z^{2}$,
$x^{3}y^{5}z$,  $x^{4}zt^{2}$, $x^{4}z^{2}w$, $x^{4}y^{2}t^{2}$, $x^{4}y^{2}zw$,
$x^{4}y^{4}w$, $x^{5}z^{2}t$,  $x^{5}y^{2}zt$, $x^{5}y^{4}t$, $x^{6}z^{3}$,
$x^{6}ytw$, $x^{6}y^{2}z^{2}$, $x^{6}y^{4}z$, $x^{7}yzw$,
$x^{7}y^{3}w$, $x^{8}yzt$, $x^{8}y^{3}t$, $x^{9}tw$, $x^{9}yz^{2}$,
$x^{9}y^{3}z$, $x^{9}y^{5}$, $x^{10}t^{2}$, $x^{10}zw$, $x^{10}y^{2}w$, $x^{11}zt$
$x^{11}y^{2}t$, $x^{12}z^{2}$, $x^{12}y^{2}z$, $x^{12}y^{4}$, $x^{13}yw$, $x^{14}yt$,
$x^{15}yz$, $x^{15}y^{3}$, $x^{16}w$, $x^{17}t$, $x^{18}z$, $x^{18}y^{2}$, $x^{21}y$, $x^{24}$
  \end{center}
\smallskip
\end{minipage}\\
\hline
66 &   \begin{minipage}[m]{.03\linewidth}\setlength{\unitlength}{.20mm}
  \smallskip
  \begin{center}
$27$
\end{center}
\smallskip
\end{minipage}
&
  \begin{minipage}[m]{.87\linewidth}\setlength{\unitlength}{.20mm}
  \smallskip
  \begin{center}
$w^{3}$, $zt^{3}$, $z^{3}w$, $yztw$, $y^{3}z^{2}$, $y^{4}t$, $xz^{2}t^{2}$,
$xyz^{2}w$, $xy^{2}tw$,  $x^{2}z^{3}t$, $x^{2}yzt^{2}$,
$x^{2}y^{2}zw$,  $x^{3}z^{4}$, $x^{3}yz^{2}t$, $x^{3}y^{2}t^{2}$,
$x^{3}y^{3}w$, $x^{4}t^{2}w$,  $x^{4}yz^{3}$, $x^{4}y^{2}zt$, $x^{5}ztw$,
$x^{5}y^{2}z^{2}$,  $x^{6}z^{2}w$, $x^{6}ytw$, $x^{6}y^{3}z$,
$x^{7}zt^{2}$, $x^{7}yzw$,  $x^{8}z^{2}t$, $x^{8}yt^{2}$, $x^{8}y^{2}w$,
 $x^{9}z^{3}$, $x^{9}yzt$, $x^{10}yz^{2}$, $x^{10}y^{2}t$, $x^{11}tw$,
$x^{11}y^{2}z$, $x^{12}zw$, $x^{12}y^{3}$, $x^{13}t^{2}$, $x^{13}yw$, $x^{14}zt$,
$x^{15}z^{2}$, $x^{15}yt$, $x^{16}yz$, $x^{17}y^{2}$,
$x^{18}w$, $x^{20}t$, $x^{21}z$, $x^{22}y$, $x^{27}$
  \end{center}
\smallskip
\end{minipage}\\
\hline
84 &  \begin{minipage}[m]{.03\linewidth}\setlength{\unitlength}{.20mm}
  \smallskip
  \begin{center}
$36$
\end{center}
\smallskip
\end{minipage}
&
  \begin{minipage}[m]{.87\linewidth}\setlength{\unitlength}{.20mm}
  \smallskip
  \begin{center}
$w^{3}$, $t^{4}$, $z^{3}w$, $yztw$, $y^{4}z$, $xyz^{2}w$, $xy^{2}tw$,  $x^{2}z^{2}t^{2}$, $x^{2}y^{2}zw$,
 $x^{3}z^{3}t$, $x^{3}yzt^{2}$, $x^{3}y^{3}w$,  $x^{4}z^{4}$, $x^{4}yz^{2}t$, $x^{4}y^{2}t^{2}$,
 $x^{5}yz^{3}$, $x^{5}y^{2}zt$, $x^{6}t^{2}w$, $x^{6}y^{2}z^{2}$, $x^{6}y^{3}t$, $x^{7}ztw$,
$x^{8}z^{2}w$, $x^{8}ytw$, $x^{9}yzw$, $x^{10}zt^{2}$, $x^{10}y^{2}w$, $x^{11}z^{2}t$, $x^{11}yt^{2}$,
  $x^{12}z^{3}$, $x^{12}yzt$, $x^{13}yz^{2}$, $x^{13}y^{2}t$, $x^{14}y^{2}z$, $x^{15}tw$, $x^{15}y^{3}$, $x^{16}zw$, $x^{17}yw$,
$x^{18}t^{2}$, $x^{19}zt$, $x^{20}z^{2}$, $x^{20}yt$, $x^{21}yz$, $x^{22}y^{2}$, $x^{24}w$, $x^{27}t$, $x^{28}z$, $x^{29}y$, $x^{36}$
   \end{center}
\smallskip
\end{minipage}\\
\hline
\end{longtable}
Using this, we see that the automorphism group of a general hypersurface in each of the families \textnumero 39, \textnumero 49, \textnumero 59, \textnumero 66, \textnumero 84 is trivial.
In particular, it does not contain any involutions. Therefore, $\mathrm{d}(X)=3$.

\section{Proof of Main Theorem}
\label{section:main}

Let us use assumptions and notations of Section~\ref{section:intro}. Suppose, in addition, that $I=1$.
If $X_d$ is a general member of the families
\textnumero 1, \textnumero 19, \textnumero 28, \textnumero 39, \textnumero 49, \textnumero 59, \textnumero 66, \textnumero 84,
then it follows from the results proved earlier in Sections~\ref{section:1}, \ref{section:19-28}, and  \ref{section:the other}  that $\mathrm{deg}(X_d)=3$.
Thus, to prove Main Theorem, we may assume that $X_d$ is not contained is the deformation families
\textnumero 1, \textnumero 19, \textnumero 28, \textnumero 39, \textnumero 49, \textnumero 59, \textnumero 66, \textnumero 84.
Let us show that $\mathrm{d}(X_{d})=2$.

We know from \cite{CheltsovPark2017, CPR, Is80b} that $X_d$ is irrational.
Thus, we have $\mathrm{d}(X_{d})\geqslant 2$.
Moreover, the natural projection
$$
\mathbb{P}(1,a_{2},a_{3},a_{4},a_{5})\dasharrow\mathbb{P}(1,a_{2},a_{3},a_{4})
$$
induces a dominant map
$X_{d}\dasharrow\mathbb{P}(1,a_{2},a_{3},a_{4})$ that is generically two-to-one if $d<3a_{5}$.
Thus, if $d<3a_{5}$, then $\mathrm{d}(X_{d})=2$.

Hence, to complete the proof of Main Theorem, we may further assume that $d\geqslant 3a_5$.
Then the 3-fold $X_d$ is contained in one of the deformation families \textnumero 9, \textnumero  17, \textnumero  27,
and the septuple  $(a_1,a_{2},a_{3},a_{4},a_{5},d,I)$ is described in the following table below.
\begin{center}
\renewcommand\arraystretch{1.4}
\begin{tabular}{|c||c|c|c|c|c|c|c|}
\hline \textnumero & $a_1$ & $a_2$ & $a_3$ & $a_4$ & $a_5$ & $d$ & $I$\\
\hline\hline
9 & $1$ & $1$ & $2$ & $3$ & $3$ & $9$ & $1$\\
\hline
17 & $1$ & $1$ & $3$ & $4$ & $4$ & $12$ & $1$\\
\hline
27 & $1$ & $2$ & $3$ & $5$ & $5$ & $15$ & $1$ \\
\hline
\end{tabular}
\end{center}

In each of these three cases, we can change quasi-homogeneous coordinates on $\mathbb{P}(1,a_{2},a_{3},a_{4},a_5)$ such that
the hypersurface $X_d$ is given by
\begin{multline*}
\quad \quad \quad \quad wt(w-t)+w^2f_{a_5}(x,y,z)+t^2g_{a_5}(x,y,z)+wth_{a_5}(x,y,z)+\\
+wf_{2a_5}(x,y,z)+tg_{2a_5}(x,y,z)+h_d(x,y,z)=0,\quad \quad \quad \quad
\end{multline*}
where $f_\ell$, $g_\ell$, $h_\ell$ are quasihomogeneous polynomials of degrees $\ell$ in $x, y, z$.
Note that $X_d$ has three singular points of type $\frac{1}{a_5}(1,a_2, a_3)$ at
$[0:0:0:0:1]$, $[0:0:0:1:0]$,  and $[0:0:0:1:1]$.
The map $X_{d}\dasharrow\mathbb{P}(1,a_{2},a_{3},a_{4})$
given by $(x:y:z:t:w)\mapsto(x:y:z:t)$ is a dominant rational map that is undefined in $[0:0:0:0:1]$.
Moreover, this map is generically two-to-one, which implies that $\mathrm{d}(X_{d})=2$.

\begin{remark}
\label{remark:elliptic-fibration}
Let $\eta\colon\mathbb{P}(1,a_{2},a_{3},a_{4},a_{5})\dasharrow\mathbb{P}(1,a_{2},a_{3})$ be the natural projection.
Then, arguing as in \cite{Cheltsov2007}, we can resolve
the indeterminacy of the rational map $\eta$ via the following commutative diagram
$$
\xymatrix{
U_{1}\ar@{->}[d]_{\pi_1}&&U_{2}\ar@{->}[ll]_{\pi_{2}}&&U\ar@{->}[ll]_{\pi_{3}}\ar@{->}[d]^{\zeta}\\
X_d\ar@{-->}[rrrr]_{\eta}&&&&\mathbb{P}(1, a_2, a_3)}
$$
where $\pi_1$ is the weighted blow up of the point $[0:0:0:0:1]$ with weights $(1, a_2, a_3)$,
$\pi_{2}$ is the weighted blow up with  weights $(1, a_2, a_3)$ of the point in $U_1$ corresponding to $[0:0:0:1:0]$,
and $\pi_{3}$ is the weighted blow up with  weights $(1, a_2, a_3)$ of the point in $U_{2}$ corresponding to $[0:0:0:1:1]$.
Then $\zeta$ is an elliptic fibration, since
the fibre of $\eta$ over a general point on $\mathbb{P}(1,a_{2},a_{3})$ is isomorphic to a smooth cubic curve on $\mathbb{P}^2$.
Furthermore, the exceptional surface of the weighted blow up $\pi_i$, which is isomorphic to $\mathbb{P}(1, a_2, a_3)$,
yields a section of the elliptic fibration.
This again implies that $\mathrm{d}(X_{d})=2$.
\end{remark}

\end{document}